\documentclass{article}
\usepackage{graphicx}
\usepackage{color}

\usepackage[fleqn]{amsmath}

\usepackage{amsthm}
\usepackage{amssymb}
\usepackage{amsbsy}
\usepackage{amsfonts}
\usepackage{mathrsfs}
\usepackage{amsmath}
\usepackage[all]{xy}
\usepackage{amstext}
\usepackage{amscd}
\usepackage[dvips]{epsfig}
\usepackage{psfrag}
\usepackage{enumerate}
\usepackage{flafter}
\allowdisplaybreaks

\theoremstyle{plain}
\newtheorem{thm}{Theorem}[section]
\newtheorem{prop}[thm]{Proposition}

\newtheorem{lem}[thm]{Lemma}
\theoremstyle{definition}
\newtheorem{exa}[thm]{Example}

\newtheorem{rem}[thm]{Remark}
\newtheorem{defn}[thm]{Definition}

\def\dim{\mathop{\mathrm{dim}}\nolimits}

\def\Qer{\mathop{\mathrm{Ker}}\nolimits}

\def\Hom{\mathop{\mathrm{Hom}}\nolimits}

\newcommand{\lra}{\longrightarrow}
\newcommand{\ra}{\rightarrow}
\newcommand{\N}{{\Bbb N}}
\newcommand{\Q}{{\Bbb Q}}
\newcommand{\R}{{\Bbb R}}
\newcommand{\Z}{{\Bbb Z}}
\newcommand{\C}{{\Bbb C}}
\newcommand{\K}{{\Bbb K}}

\newcommand{\Aut}{{\rm Aut}}

\newcommand{\GL}{{\rm GL}}
\newcommand{\Ker}{{\rm Ker}}

\begin{document}
\large

\begin{center}{\bf  \Large A solvable extended logarithm of the Johnson homomorphism} \end{center}
\begin{center}{Takefumi Nosaka\footnote{
E-mail address: {\tt nosaka@math.titech.ac.jp}
}}\end{center}
\begin{abstract}\baselineskip=12pt \noindent
Concerning Johnson's homomorphism from the Torelli group, there are previous works to define a logarithm of the homomorphism,
and give some extension of the logarithm. 
This paper considers exponential solvable elements in the mapping class group of a surface,
and defines the logarithms of such elements.
% suggests another extension of the logarithm in terms of solvable Lie groups. H
%We suggest an extension of a certain logarithm of the total Johnson map in terms of solvable Lie groups. Here, the domain of the map is extended to a subset consisting of exponential solvable elements in the mapping class group of a surface.
\end{abstract}
\begin{center}
\normalsize
\baselineskip=11pt
{\bf Keywords} \\
\ \ \ mapping class group, exponential solvable Lie group, knot \ \ \ 
\end{center}
\begin{center}
\normalsize
\baselineskip=11pt
{\bf Subject Codes } \\
\ \ \ 14J50, 22E25, 16N40, 58K15, 22E25 \ \ \
\end{center}

\large
\baselineskip=16pt
\section{Introduction}
\label{S00}
Let $\Sigma_{g,1}$ be an orientable compact surface of genus $g >0$ with one boundary component and $\mathcal{M}_{g,1}$ be the mapping class group of $\Sigma_{g,1}$ relative to the boundary. Since $\mathcal{M}_{g,1} $ acts on the homology $ H_1(\Sigma_{g,1} ;\Z)\cong \Z^{2g}$, we have a homomorphism $\mathcal{P}: \mathcal{M}_{g,1} \ra \GL_{2g}(\Z)$; the kernel is called the Torelli group $\mathcal{I}_{g,1}$. Let $\mathcal{L} $ be the free Lie algebra of rank $2g$ over $\Q$.

The Johnson homomorphism \cite{DJ, Mo} whose domain is the Torelli group $\mathcal{I}_{g,1}$ has been a source of insight into $\mathcal{M}_{g,1} $ in terms of nilpotent Lie algebras and is extended to the total Johnson map \cite{Ka1}. % from $\mathcal{M}_{g,1}$.
To describe this map, we prepare the lower central series $\mathcal{L} = \mathcal{L}_1 \supset \mathcal{L}_2 \supset \cdots $, and denote the inverse limit $\lim_{\infty \leftarrow k} \mathcal{L} / \mathcal{L} _{k} $ by $\mathcal{L}^{\rm nil}$. We also take the automorphism group, $\Aut_{\rm Lie}(\mathcal{L}^{\rm nil})$ of the Lie algebra $\mathcal{L}^{\rm nil}$, and a certain subgroup $\Aut_{\rm 0}(\mathcal{L}^{\rm nil})$, which preserves a symplectic structure; see \eqref{seq37} for details. Then, %as is explained in Section \ref{33773}, 
the total Johnson map\footnote{More explicitly, as mentioned in Section \ref{33773}, we need to choose a symplectic Magnus expansion to define the map.} is a map from $\mathcal{M}_{g,1}$ to the subgroup $\Aut_{\rm 0}(\mathcal{L}^{\rm nil})$. Meanwhile, this $\Aut_{\rm 0}(\mathcal{L}^{\rm nil})$ has a Lie algebra isomorphic to the set, $\mathrm{Der}_{\omega}(\mathcal{L}^{\rm nil} )$, consisting of derivations of $\mathcal{L}^{\rm nil}$ which annihilate a symplectic form $\omega$. 
The set $\mathrm{Der}_{\omega }(\mathcal{L}^{\rm nil} )$ has been widely studied as a Lie algebra of the Lie wor(l)d in the sense of Kontsevich \cite{Kon}. %, and from several approaches. %symplectic automorphism group. 
Since the Goldman Lie algebra of $\Sigma_{g,1}$ canonically injects $\mathrm{Der}_{\omega }(\mathcal{L}^{\rm nil} )$ (see \cite{KK, KK2} for details), if we can define a logarithm from $\Aut_{\rm 0}(\mathcal{L}^{\rm nil}) $ to $\mathrm{Der}_{\omega}(\mathcal{L}^{\rm nil} ) $, then the total Johnson map serves as a bridge from $\mathcal{M}_{g,1} $ to the Goldman Lie algebra.

Let us briefly explain the previous results of such logarithms. 
%However, the known domains of the logarithm have ambiguous aspects. 
In \cite{KK}, some logarithms are first constructed from the images of $\mathcal{I}_{g,1}$ and the Dehn twists in $\mathcal{M}_{g,1}$. More generally, the paper \cite{MT} % gives an extension of the domain in $\Aut_{\rm Lie}(\mathcal{L}^{\rm nil})$ and
defines a generalization of (the logarithm of) Dehn twists. However, any existing element $\phi \in \mathcal{M}_{g,1}$ which defines a logarithm has the property that every eigenvalue of $\mathcal{P}(\phi)$ is one.

In this paper, we suggest an extension of the logarithms in terms of solvable Lie groups over $\C$. Here, for a matrix $A \in \GL_n(\C)$ with eigenvalues $\lambda_1, \dots, \lambda_{n} \in \C^{\times}$, we define $\mathrm{Eig}(A) $ to be the multiplicative subgroup of $ \C^{\times}$ generated by $\lambda_1, \dots, \lambda_{n}, \bar{\lambda_1}, \dots, \bar{\lambda_{n}}$; we say that $A$ is {\it exponential solvable} if $\mathrm{Eig}(A) \cap \{ z \in \C \mid 1= |z|\}=\{1\} $. Our main theorem (Theorem \ref{thm2}) states that if $\phi \in \mathcal{M}_{g,1} $  such that $\mathcal{P}(\phi)$ is an exponential solvable matrix, we can define a logarithm of $\phi$. More generally, we discuss the existence of logarithms of exponential solvable elements of $\Aut_{\rm Lie}(\mathcal{L}^{\rm nil})$ (see Theorem \ref{thm1}); we show these logarithms to be extensions of the ones in \cite{KK, MT}; see Theorem \ref{thm7} and \eqref{seq3}.

While there have been many nilpotent approaches to $\mathcal{M}_{g,1}$, this paper suggests an approach from solvability. The point of constructing logarithms is that, starting from any exponential solvable $\phi \in \Aut_{\rm Lie}(\mathcal{L}^{\rm nil})$, we can find a solvable Lie group such that the exponential map $\mathrm{Exp}$ is diffeomorphic on its image and $\phi$ lies in its image; accordingly, we can define the logarithm of $\phi$ to be $\mathrm{Exp}^{-1}(\phi)$; see Sections \ref{333}--\ref{99se33} for the details. Furthermore, we suggest a procedure for computing $\mathrm{Exp}^{-1}(\phi)$ concretely in terms of the Baker-Campbell-Hausdorff (BCH) formula; see Section \ref{99414}. In Section \ref{33773}, we discuss little applications of the logarithms to topological invariants.

\

\noindent
{\bf Conventional notation.} \ Let $\K$ be one of the fields $\Q,\R,\C$. By $\mathcal{L}$, we mean the free Lie algebra of rank $n \in \mathbb{N}$ over $\Q$.

\section{Main theorems and applications}
\label{S1}
The purpose of this section is to state the main theorems (the proofs are given in Section \ref{99se33}). For this, we first introduce base matrices of an invertible $(n \times n)$-matrix $A$ over $\C$. Consider the exponential map $\exp: \mathrm{Mat}(n \times n ; \C ) \ra \GL_n( \C ) $ that sends $B$ to $\sum_{i=0}^{\infty} B^i / i !$. We call an $(n \times n)$-matrix $B$ satisfying $\exp (B)=A $ {\it a base matrix of $A$}.
\begin{exa}\label{b3b7}
If we fix a branch of the logarithm function on $\C^\times $, we can uniquely choose a base matrix for every $A\in \GL_n( \C ) $ as follows. Let $\ln: \C^{\times } \ra \C$ be the map which takes $z$ to $\ln (|z|) + \sqrt{-1} \mathrm{arg} (z)$, where $- \pi < \mathrm{arg} (z) \leq \pi $.

First, we suppose the case where $g \in \GL_n( \C ) $ exists such that $g^{-1}A g=J_{\lambda}(n) $, where $J_{\lambda}(\ell) \in \mathrm{Mat}( \ell \times \ell, \C)$ is the Jordan block of size $\ell$ and eigenvalue $\lambda $. Then, we define
$$ \ln (A) := \ln( \lambda ) I_\ell + \sum_{i: 1 \leq i < \ell } (-1)^{i-1} \frac{(\lambda^{-1} A - I_\ell)^i }{i } .$$
In general, if $g \in \GL(n , \C ) $ exists such that $g^{-1}A g=J_{\lambda_1}(\ell_1) \oplus \cdots \oplus J_{\lambda_m}(\ell_m) $ by Jordan decomposition, we define $ \ln (A) $ to be $g \bigl( \ln( J_{\lambda_1}(\ell_1)) \oplus \cdots \oplus \ln(J_{\lambda_m}(\ell_m) )\bigr) g^{-1} $. Accordingly, we can easily check that the definition is independent of the choice of $g$ and $\exp( \ln (A))=A$, i.e., $\ln(A)$ is a base matrix of $A$.
Furthermore, the uniqueness can be shown as an elementary exercise of linear algebra.
\end{exa}
Furthermore, we need some terminology for the derivations. Let $\mathfrak{g} $ be a real Lie algebra of finite dimension and $\mathrm{Der}( \mathfrak{g})$ be the Lie algebra consisting of derivations on $ \mathfrak{g}$, that is,
$$\mathrm{Der}( \mathfrak{g})= \{f \in \Hom(\mathfrak{g} , \mathfrak{g}) \mid f ([x,y])= [f(x),y]+[x,f(y)] \ \ \mathrm{for \ any \ }x,y \in \mathfrak{g} \} .$$
In addition, let us define $ \Aut_{\rm Lie}(\mathfrak{g} )$ to be the group consisting of Lie algebra automorphisms of $\mathfrak{g} $. It is a Lie group, since it is a closed subgroup of $\GL(\mathfrak{g}) $. Then, as is classically known (see \cite[Proposition 7.3]{SW}), the Lie algebra of $ \Aut_{\rm Lie}(\mathfrak{g} )$ coincides with $\mathrm{Der}( \mathfrak{g}) $, and the exponential map $\exp: \mathrm{Der}( \mathfrak{g}) \ra \mathrm{Aut}_{\rm Lie}( \mathfrak{g})$ is given by $ D \mapsto \sum_{j=0}^{\infty} D^j/j!$. Let $ \mathfrak{g}_{\rm ab}$ be the abelianization of $\mathfrak{g}$, i.e., $ \mathfrak{g}_{\rm ab}= \mathfrak{g}/[ \mathfrak{g}, \mathfrak{g}]$. 
For $k \in \mathbb{N}$, recall $\mathcal{L}_k  \subset \mathcal{L}$ from the introduction. 
If $ \mathfrak{g} $ is the quotient $ \mathcal{L}/\mathcal{L}_k$, let us define projections
\[ p_k:\mathrm{Aut}_{\rm Lie}( \mathfrak{g}) \ra \mathrm{Aut}_{\rm Lie}( \mathfrak{g}_{\rm ab})= \GL(\mathfrak{g}_{\rm ab}) , \ \ \ \ p_k :\mathrm{Der}( \mathfrak{g}) \ra \mathrm{Der}( \mathfrak{g}_{\rm ab}) = \mathrm{End}(\mathfrak{g}_{\rm ab} ) \]
to be the induced maps from the projection $\mathfrak{g} =  \mathcal{L}/\mathcal{L}_k  \ra \mathfrak{g}/[ \mathfrak{g}, \mathfrak{g}]$.
The main theorems are for the case in which $\mathfrak{g} $ is the complexification of a free nilpotent Lie algebra. More precisely, 
\begin{thm}\label{thm1} 
Let $ \mathcal{L}/\mathcal{L}_k $ be the $k$-th nilpotent quotient of the free Lie algebra $ \mathcal{L}$ of finite rank over $\Q$ and $\mathfrak{g}$ be the complexification $\mathcal{L}/\mathcal{L} _k \otimes \C $. For $ \Phi \in \mathrm{Aut}_{\rm Lie}( \mathfrak{g})$, we choose a base matrix, $\phi$, of $p_k(\Phi)$. Suppose that $p_k(\Phi ) \in \mathrm{Aut}( \mathfrak{g}_{\rm ab}) $ is exponential solvable, that is, $\mathrm{Eig}(p_k(\Phi ) ) \cap \{ z \in \C \mid 1= |z|\}=\{1\} $ (see the introduction for the definition of $\mathrm{Eig}(p_k(\Phi ) )$).

Then, there is a unique derivation $D_\Phi \in \mathrm{Der}( \mathfrak{g}) $ such that $\exp(D_\Phi )=\Phi $ and $ p_k(D_\Phi )= \phi$.
\end{thm}
\begin{rem}\label{rrhm1} 
The assumption of exponential solvability is necessary. For example, if $ p_k(\Phi ) = \mu_m I_n$ for some primitive $m$-th root of unity, we can not subsequently apply Lemma \ref{thm10}, which plays a key role in the proof. In particular, it seems hard to define a logarithm of the hyper-elliptic involution in $\mathcal{M}_{g,1}$. 
\end{rem}
\noindent
In terms of the base matrices in Example \ref{b3b7}, we introduce the natural logarithms of $\mathrm{Aut}_{\rm Lie}( \mathfrak{g})$:
\begin{defn}\label{def2} 
For $ \Phi \in \mathrm{Aut}_{\rm Lie}( \mathfrak{g})$ such that $p_k(\Phi)$ is exponential solvable, we define $\ln (\Phi)$ to be the derivation $ D_\Phi $ in the above theorem satisfying $ p_k(D_\Phi )= \ln (p_k( \Phi))$.
\end{defn}
In addition, let us discuss the symplecticity of $\Phi$. Suppose $n=2g$ for some $g \in \N$. Take a basis, $x_1, x_2, \dots, x_{2g}$, of $ \mathcal{L}$, and define $\omega := \sum_{i=1}^g [x_{2i-1},x_{2i}] \in \mathcal{L}/ \mathcal{L}_k$. As in \cite{GL} and \cite[Definition 3.4]{Mo2}, we define two groups by setting
\[ \Aut_0 (\mathcal{L}/ \mathcal{L}_k \otimes \K )' := \{ \Phi \in \Aut_{\rm Lie} (\mathcal{L}/ \mathcal{L}_k \otimes \K )\ | \ \Phi( \omega ) = \omega \}, \]
\begin{equation}\label{seq37}
\Aut_0 (\mathcal{L}/ \mathcal{L}_k \otimes \K ):= q_{k+1}(\Aut_0 (\mathcal{L}/ \mathcal{L}_{k+1} \otimes \K )' ),
\end{equation}
where $q_{k+1}$ is the projection $\Aut_{\rm Lie} (\mathcal{L}/ \mathcal{L}_{k+1} \otimes \K) \ra \Aut_{\rm Lie} (\mathcal{L}/ \mathcal{L}_{k} \otimes \K)$. In parallel, we define the Lie subalgebra, $\mathrm{Der}_{\omega} ( \mathcal{L}/ \mathcal{L}_k \otimes \K) $, to be $\{ D \in \mathrm{Der} ( \mathcal{L}/ \mathcal{L}_k \otimes \K) \mid D(\omega)=0 \}$. 
\begin{thm}\label{thm2}
Let $\mathfrak{g}$ be the complexification $\mathcal{L}/\mathcal{L} _k \otimes \C $ as in Theorem \ref{thm1}. Take $ \Phi \in \mathrm{Aut}_{\rm Lie}( \mathfrak{g})$. If we can choose an exponential solvable base matrix, $\phi$, of $p_k(\Phi)$, and if $\Phi$ lies in $ \mathrm{Aut}_{0}( \mathfrak{g} ) $, then the derivation $D_\Phi $ claimed in Theorem \ref{thm1} lies in the subalgebra $ \mathrm{Der}_{\omega} ( \mathfrak{g}).$
\end{thm}
Furthermore, let us discuss the formal Maclaurin expansion and convergence. For $ \Phi \in \mathrm{Aut}_{\rm Lie}( \mathcal{L}/\mathcal{L} _k \otimes \K)$, consider the formal sum $\mathrm{Log}(\Phi) := -\sum_{ i=1}^{\infty} ( \mathrm{id} - \Phi )^i/i$. As in \cite[Sections 4 and 7]{KK2}, let us define $\mathcal{M}^{\circ }_{k,\K}$ to be the subset of $ \mathrm{Aut}_{\rm Lie}( \mathcal{L}/\mathcal{L} _k \otimes \K ) $ consisting of $ \Phi$'s such that $ \mathrm{Log}(\Phi)$ converges in $ \mathrm{Der} ( \mathcal{L}/\mathcal{L} _k )\otimes \K $. We further define $\mathcal{M}^{\circ }_{\infty, \K}$ to be the limit $\displaystyle{\lim_{\infty \leftarrow k }} \mathcal{M}^{\circ }_{k,\K} \subset \mathrm{Aut}_{\rm Lie}( \displaystyle{\lim_{\infty \leftarrow k }} \mathcal{L}/\mathcal{L} _k \otimes \K ) $. %subset of $ \mathrm{Aut}_{\rm Lie}( \displaystyle{\lim_{\infty \leftarrow k }} \mathcal{L}/\mathcal{L} _k \otimes \K ) $ with the convergence of the formal sum. 
For example, as shown in \cite{KK}, if $\Phi$ is an element derived from the Torelli group or Dehn twists in $\mathcal{M}_{g,1}$, $\Phi $ lies in $\mathcal{M}^{\circ }_{\infty,\Q} $; see also \cite[Sections 4 and 8]{MT} for a generalization of these examples; however, as mentioned in Section \ref{S00}, all of the existing examples of $\Phi \in \mathcal{M}^{\circ }_{k,\K} $ are such that every eigenvalue of $p_k(\Phi) \in \GL_n(\C)$ is one (see \cite[Section 4.3]{MT}), and the set $\mathcal{M}^{\circ }_{\K} $ is considered to be far from being explicitly determined.

Now, let us establish a generalization of these results:

\begin{thm}\label{thm7}%Let $\K=\C. $
If $\Phi \in \mathrm{Aut}_{\rm Lie}( \mathcal{L}/\mathcal{L} _k \otimes \K) $ satisfies that every eigenvalue of $p_k(\Phi)$ is one, then $\Phi$ lies in $\mathcal{M}^{\circ }_{k,\K} $. Moreover, if $\K=\C$, the equality $\mathrm{Log}(\Phi) = \mathrm{ln}(\Phi) $ holds in $ \mathrm{Der}(\mathcal{L}/\mathcal{L} _k \otimes \C ) $. This statement holds even considering the inverse limit according to $k \ra \infty.$
\end{thm}
Finally, we now give a similar theorem in terms of Hopf algebras. 
According to \cite{Kon}, this discussion is regarded in the associative context. Let $\hat{T}$ be the ring $ \K[\![ X_1, \dots, X_n ]\!]$ of formal power series with non-commutative variables $X_1, \dots, X_n$, and $\hat{T}_k$ be the ideal of $ \hat{T}$ generated by formal power series of degree $\geq k$. Let $ \hat{T}_{\infty}$ be $\{0\}$. For $k \in \mathbb{N} \cup \{\infty \}$, let us consider the coproduct map $\hat{\Delta}: \hat{T}/ \hat{T}_k \ra \hat{T}/ \hat{T}_k \otimes\hat{T}/\hat{T}_k $ defined by $ \hat{\Delta} (X_i)= 1 \otimes X_i +X_i\otimes 1$, where $\otimes$ with $k= \infty$ is defined to be the complete tensor product as in \cite[Appendix A.1]{Q}. Take the following two automorphism groups: 
\begin{eqnarray}
\Aut_{\rm alg} ( \hat{T}/\hat{T}_k ) &:=& \{ \K \textrm{-algebra \ automorphism} \ U :\hat{T}/\hat{T}_k  \ra \hat{T}/\hat{T}_k    \mid \ U(\hat{T}_{p}) \subset \hat{T}_{p} \mathrm{ \ for \ any \ } p \in \N \}, \notag \\
\Aut_{\rm Hopf} ( \hat{T}/\hat{T}_k )&:=& \{ U \in \Aut_{\rm alg} ( \hat{T}/\hat{T}_k ) \mid U \mathrm{ \ preserves \ } \hat{\Delta} \} .
\end{eqnarray}
As is well-known, the set of primitive elements of $\hat{T}/\hat{T}_k $,
$$\mathcal{P}( \hat{T}/\hat{T}_k) :=\{ x \in \hat{T}/\hat{T}_k \mid
\hat{\Delta} (x)= 1 \otimes x +x\otimes 1 \} ,$$
has a Lie algebra structure with $[X,Y]=X Y -Y X$ and is isomorphic to $ \mathcal{L}/\mathcal{L}_k$. According to \cite[Theorem A.3.3]{Q}, the restrictions of Hopf automorphisms of $\hat{T}/\hat{T}_k $ give rise to a group isomorphism $ \Aut_{\rm Hopf} ( \hat{T}/\hat{T}_k ) \cong \Aut_{\rm Lie} ( \mathcal{L}/\mathcal{L}_k ) $. Thus, $ \Aut_{\rm Lie} ( \mathcal{L}/\mathcal{L}_k ) $ is regarded as a subgroup of $\Aut_{\rm alg} ( \hat{T}/\hat{T}_k )  $. 
\begin{thm}\label{thm51}
Let $\mathrm{Der} ( \hat{T}/\hat{T}_k \otimes \C)$ be the set of all derivations of $\hat{T}/\hat{T}_k \otimes \C$, and 
$p_k : \Aut_{\rm alg} ( \hat{T}/\hat{T}_k \otimes \C)  \ra  \Aut_{\rm alg} ( \hat{T}/\hat{T}_2 \otimes \C)= \GL_n(\C) $ be the projection. 
For $ \Phi \in \Aut_{\rm alg} ( \hat{T}/\hat{T}_k ) $, we choose a base matrix, $\phi$, of $p_k(\Phi)$. If $p_k(\Phi ) \in \GL_n(\C) $ is exponential solvable, then there is a unique derivation $D_\Phi \in \mathrm{Der} ( \hat{T}/\hat{T}_k \otimes \C)$ such that $\exp(D_\Phi )=\Phi $ and $ p_k(D_\Phi )= \phi$.
\end{thm}

%we may focus on the $\Aut_{\rm Hopf} ( \hat{T}/\hat{T}_k ) $ instead of $ \Aut_{\rm Lie} ( \mathcal{L}/\mathcal{L}_k ) $.

The above theorems can be summarized as a commutative diagram,
\begin{equation}\label{seq3}
{ \xymatrix{
\ln: \{ \Phi \in \Aut_{\rm alg} ( \hat{T}/\hat{T}_k \otimes \C ) \mid p_k(\Phi) \textrm{ is exponential solvable} \} \ar[r]\ar@{}[d]|{\bigcup}& \mathrm{Der} ( \hat{T}/\hat{T}_k \otimes \C)\ar@{}[d]|{\bigcup}\\
\ln: \{ \Phi \in \mathrm{Aut}_{\rm Lie}( \mathcal{L}/\mathcal{L} _k \otimes \C ) \mid p_k(\Phi) \textrm{ is exponential solvable} \} \ar[r]\ar@{}[d]|{\bigcup}& \mathrm{Der} ( \mathcal{L}/\mathcal{L} _k \otimes \C)\ar@{}[d]|{\bigcup}\\
\ln: \{ \Phi \in \mathrm{Aut}_{\rm 0}( \mathcal{L}/\mathcal{L} _k \otimes \C ) \mid p_k(\Phi) \textrm{ is exponential solvable} \} \ar[r]\ar@{}[d]|{\bigcup}& \mathrm{Der}_{\omega} ( \mathcal{L}/\mathcal{L} _k \otimes \C)\ar@{=}[d] \\
\mathcal{I}_{g,1} \cup \{ \mathrm{Dehn \ twists}\} \ar[r]^{\mathrm{Log}} & \mathrm{Der}_{\omega} ( \mathcal{L}/\mathcal{L} _k \otimes \C). \\
}}\end{equation}
%\begin{equation}\label{seq1331}\ln: \{ \Phi \in \mathrm{Aut}_{\rm Lie}( \mathcal{L}/\mathcal{L} _k \otimes \C ) \mid p_k(\Phi) \textrm{ is exponential solvable} \} \lra \mathrm{Der} ( \mathfrak{g}) ; \ \ \Phi \longmapsto \ln (\Phi). \end{equation}

\section{Computations of the natural logarithm $ \mathrm{ln}(\Phi) $}
\label{99414}
This section shows a procedure for computing the natural logarithm $ \mathrm{ln}(\Phi) $ from the viewpoints of generalized Magnus expansions and rational homotopy theory \cite{Q}.
Recall the isomorphism $ \Aut_{\rm Hopf} ( \hat{T}/\hat{T}_k ) \cong \Aut_{\rm Lie} ( \mathcal{L}/\mathcal{L}_k ) $ from the previous section. Thus, to investigate the above groups $\Aut_{\rm Lie} ( \mathcal{L}/\mathcal{L}_k ) $ and $ \Aut_{\omega} ( \mathcal{L}/\mathcal{L}_k )$, we may focus on $\Aut_{\rm Hopf} ( \hat{T}/\hat{T}_k ) $ hereafter. %T instead of $ \Aut_{\rm Lie} ( \mathcal{L}/\mathcal{L}_k ) $.

As is done in \cite[Section 1]{Ka1}, we first observe the decompositions \eqref{777} below. Recall the lemma:
\begin{lem}
[{\cite[Lemma 1.2]{Ka1}}]\label{thm130} Take $k \in \mathbb{N} \cup \{\infty \}$. A $\K$-algebra endomorphism $U$ of $\hat{T}/\hat{T}_k $ lies in $\Aut_{\rm alg} ( \hat{T}/\hat{T}_k ) $ if and only if $U(\hat{T}_{p}) \subset \hat{T}_{p}$ for each $p \in \N$ and the induced map of $U$ on $\hat{T}_1/\hat{T}_2 = H_{\K}$ is an isomorphism.
\end{lem}
\noindent
Let $H_{\K}$ be $\hat{T}_1/\hat{T}_2 \cong \K^n$ as a vector space. For $U \in \Aut_{\rm alg} ( \hat{T}/\hat{T}_k ) $, we denote the projection of $U$ on $\hat{T}_1/\hat{T}_2 $ by $p_k(U)\in \GL_n(\K)$ and have the following group homomorphisms,
$$ p_k : \Aut_{\rm alg} ( \hat{T}/\hat{T}_k ) \ra \GL_n(\K), \ \ \ \ \ \Aut_{\rm Hopf} ( \hat{T}/\hat{T}_k ) \ra \GL_n(\K). $$
We denote the kernels by $\mathrm{IA}_{\rm alg} ( \hat{T}/\hat{T}_k ) $ and $\mathrm{IA}_{\rm Hopf} ( \hat{T}/\hat{T}_k ) $. We will show splittings of the homomorphisms. For $ A \in \GL_n(\K)$ and $j \in \N$, consider the tensor representation, i.e., $A(z_j):= A^{\otimes j} z_j$ for $z_j \in H_{\K}^{\otimes j}$. Identifying $\hat{T}/\hat{T}_k $ with $\prod_{j=0}^{k-1} H_{\K}^{\otimes j}$, the tensor representation yields a homomorphism $\mathfrak{s}: \GL_n(\K) \ra GL( \hat{T}/\hat{T}_k )$. We can easily verify that $\mathfrak{s}$ lies in $ \Aut_{\rm alg} ( \hat{T}/\hat{T}_k ) $ and preserves $\hat{\Delta}$. Namely, $\mathfrak{s}$ is a splitting of $p_k$. Hence, we have the semi-direct products,
\begin{equation}\label{777} \Aut_{\rm alg} ( \hat{T}/\hat{T}_k ) \cong \mathrm{IA}_{\rm alg} ( \hat{T}/\hat{T}_k ) \rtimes GL(H_\K), \ \ \ \ \ \Aut_{\rm Hopf} ( \hat{T}/\hat{T}_k ) \cong \mathrm{IA}_{\rm Hopf} ( \hat{T}/\hat{T}_k ) \rtimes GL(H_\K). 
\end{equation}
To gain an understanding of the subgroup $\mathrm{IA}_{\rm alg} ( \hat{T}/\hat{T}_k ) $, let us examine the bijection \eqref{7577} below. Every $\K$-linear homomorphism $\alpha: H_{\K} \ra \hat{T}_2/\hat{T}_k$ can be uniquely extended to an algebra homomorphism $\tilde{\alpha}: \hat{T}/\hat{T}_k\ra \hat{T}/\hat{T}_k$ such that the induced map on $\hat{T}/\hat{T}_2$ is $\mathrm{id}_{\hat{T}/\hat{T}_2}$.
Thus, $\tilde{\alpha}  $ an isomorphism, by Lemma \ref{thm130}. To summarize, we have a bijection\footnote{The bijection is shown in \cite[(1.7)]{Ka1}, and is equal to the inverse of $E$ in \cite{Ka1}.} 
\begin{equation}\label{7577} E: \Hom_{\K}( H_{\K}, \hat{T}_2/\hat{T}_k ) \stackrel{\sim}{\lra }\mathrm{IA}_{\rm alg} ( \hat{T}/\hat{T}_k ) ; \ \ \alpha \longmapsto  \tilde{\alpha} .\end{equation}
Moreover, the preimage of $\mathrm{IA}_{\rm Hopf} ( \hat{T}/\hat{T}_k ) $ implies
\begin{equation}\label{75778} E^{-1}( \mathrm{IA}_{\rm Hopf} ( \hat{T}/\hat{T}_k ) )= \{ f \in \Hom_{\K}( H_{\K}, \hat{T}_2/\hat{T}_k ) \mid f (x) \in \mathcal{P}( \hat{T}/\hat{T}_k) \ \mathrm{for \ any \ }x \in H_\K \}.
\end{equation}

From viewpoints of \eqref{777} and \eqref{7577}, the group structure of $\Aut_{\rm alg} ( \hat{T}/\hat{T}_k ) $ can be described as follows (The description for $k \leq 4$ appears in \cite[Lemma 1.4]{Ka1}). Let $P(m) $ be the set of partitions of $m$. For $i_1, \dots, i_\ell \in \mathbb{Z}_{\geq 1} $ with $\sum_{j=1}^\ell i_j =m$, we denote the associated partition by $[i_1| i_2 | \cdots | i_\ell] \in P(m)$. In addition, for $u \in \Hom_{\K}( H_{\K}, \hat{T}_2/\hat{T}_k ) $ and $A\in  \GL(H_{\K})$, we define $((u,A))$ to be $E(u) \circ A \in \Aut_{\rm alg} ( \hat{T}/\hat{T}_k ) $, where $A$ may lie in $ \Aut_{\rm alg} ( \hat{T}/\hat{T}_k )$ by \eqref{777}. We define $u_m $ to be the element in $ \Hom_{\K}( H_{\K}, \hat{T}_m/\hat{T}_{m+1} ) $ that is determined by 
$$ ((u,A)) a = Aa + \sum_{m : 1 \leq m < k} u_m(Aa) \in \hat{T}/\hat{T}_k $$
for any $a \in H_{\K}.$ By a straightforward computation, we can show the following:
\begin{lem}
[{cf. \cite[Lemma 1.4]{Ka1}}]\label{yyyy} Suppose $ ((w,C))=((u,A))((v,B)) \in \Aut_{\rm alg} ( \hat{T}/\hat{T}_k ) $ for $(u,A)$, $(v,B),$ and $(w,C) \in \mathrm{IA}_{\rm alg} ( \hat{T}/\hat{T}_k ) \rtimes GL(H_\K) $. Then, $C=AB$ and
$$ w_m (a)= u_m (a) + \sum_{ [i_1| i_2 | \cdots | i_\ell] \in P(m) } (u_{i_1} \otimes \cdots \otimes u_{i_\ell}) A^{\otimes \ell} \bigl( v_\ell (A^{-1}a ) \bigr) $$
hold for any $a\in H_{\K}$ and $m \geq 2$. Here, $u_1=v_1=0$.
\end{lem}
\begin{exa}\label{yyyyy} 
By abusing the notation, we will denote the operation $ a \mapsto A^{\otimes \ell} (v_{\ell}(A^{-1}a))$ by $A v_{\ell}$ and give concrete descriptions of $w_m$ for $m \leq 4$:
\[w_2 = u_2 +Av_2, \ \ \ w_3= u_3 +(u_2 \otimes 1 + 1 \otimes u_2) Av_2+ v_3, \]
\[w_4= u_4 +(u_3 \otimes 1 + 1 \otimes u_3+ u_2 \otimes u_2) Av_2+(u_2 \otimes 1 \otimes 1+ 1 \otimes u_2\otimes 1+ 1\otimes 1 \otimes u_2) Av_3+A v_4.\]
\end{exa}

Finally, for $\K=\C$, we can show a procedure for computing the logarithm $\ln(\Phi) $ for $\Phi \in \Aut_{\rm Hopf} ( \hat{T}/\hat{T}_k )\cong \Aut_{\rm Lie} ( \mathcal{L}/\mathcal{L}_k \otimes \C) $ such that $p_k(\Phi) $ is exponential solvable. Following the semi-direct products \eqref{777}, we decompose $\Phi$ as $(\mathrm{I}\Phi , p_k(\Phi)) \in \mathrm{IA}_{\rm Hopf} ( \hat{T}/\hat{T}_k ) \rtimes \GL(H_\C) $. As mentioned in Sections \ref{333} and \ref{99se33}, there uniquely exists $X\in \mathrm{Der}(\mathcal{L}/\mathcal{L}_k \otimes \C) $ satisfying $\exp(X)= (\mathrm{I}\Phi,1) $. Letting $Y \in \mathfrak{gl}_n(\C)$ be $ \ln(p_k(\Phi))$, we formally obtain the following from the BCH formula:
\[\ln ( \Phi)= \ln ( ( \mathrm{I}\Phi ,1)(1, p_k(\Phi)))= \log ( \exp(X)  \exp (Y ))=\]
\begin{equation}\label{xyx}= {\displaystyle X+Y+{\frac {1}{2}}[X,Y]+{\frac {1}{12}}[X,[X,Y]]-{\frac {1}{12}}[Y,[X,Y]]+\cdots } .
\end{equation}
\begin{lem}
[{See \S \ref{99se33} for the proof}]\label{yyy} The BCH-formula \eqref{xyx} converges in $ \mathrm{Der}(\mathcal{L}/\mathcal{L}_k \otimes \C ) $.
\end{lem}
To conclude, if we find concretely a $Y$ satisfying $\exp(Y)= (\mathrm{I}\Phi,1) $ and a decomposition $\Phi =(\mathrm{I}\Phi , p_k(\Phi)) $ with small $k \in \N$, we can sometimes compute $\ln ( \Phi) $ via the BCH formula; see \cite[Section 7]{KK} and \cite[Sections 5 and 8]{MT} for a description of computing such $Y$'s. For example, let us give a concrete computation for $k=3$, where we denote $[X, Y]$ by $\mathrm{ad}(X)(Y)$.
\begin{exa}\label{y1}
First, suppose that $k=3$. Then, any term in \eqref{xyx} which contains two instances of $Y$ vanishes. Thus, we can apply a reduction of the BCH formula (see \cite[Section 3.3]{Reu}) to $\mathrm{Im}(\Phi)$, as $X+ \mathrm{ad}(\frac{X}{1 - e^X})(Y) $; more precisely, if we define the Bernoulli numbers $b_{2n}$ by the Taylor expansion $x/(1-e^x)= 1+x/2 + \sum_{n=1}^{\infty} \frac{1}{(2n)!} b_{2n}x^{2n} $, the $\mathrm{ad}(\frac{X}{1 - e^X})(Y)$ is defined by $ Y+\frac{1}{2}[X, Y] +\sum_{n=1}^{\infty} \frac{1}{(2n)!} b_{2n} \mathrm{ad}(X)^{2n}(Y).$ For example, if $X$ is a diagonal matrix, it is not so hard to compute $\mathrm{ad}(\frac{X}{1 - e^X})(Y)$.
\end{exa}

Concerning the case $k \geq 4$, if a term in the BCH formula contains a bracket of $Y$ $(k-2)$ times, the term vanishes. Since we can find explicit descriptions of order $O( Y^k )$ in the BCH formula (see, e.g., \cite[Section 4]{ML} or \cite[Section 3.3]{Reu}), we can find an explicit formula for $ \ln ( \Phi)$, as in Example \ref{y1}.

\section{Motivation for studying the Johnson map and topological invariants}
\label{33773}
In this section, we explain our motivation behind defining an extension of the logarithm. Throughout this section, we will fix $n, k \in \mathbb{Z}_{\geq 2} \cup \{ \infty\} $, and the free group, $F$, with basis $x_1,x_2, \dots, x_n$,

Let us begin by introducing Magnus expansions. Take an invertible matrix $A = \{ a_{ij}\}_{1 \leq i,j \leq n} \in \GL_{n}(\K)$. For $m \geq 0$, we define $ (\hat{T}/\hat{T}_k)^{\times}$ to be the multiplicative group consisting of invertible elements of $ \hat{T}/\hat{T}_k$. Inspired by \cite{Ka1,KK}, we define {\it a Magnus expansion (over $A$)} to be a group homomorphism $\theta: F \ra (\hat{T}/\hat{T}_k)^{\times}$ satisfying $ \theta ( x_i)= 1+ \sum_{j=1}^n a_{ij} X_j$ modulo $ \hat{T}_2$ for any $i \leq n$. If $A$ is the identity matrix $I_n$, this definition is the same as the generalized Magnus expansion \cite{Ka1, KK}. A Magnus expansion $\theta$ is said to be {\it group-like} if $\hat{\Delta} (\theta (x)) =\theta (x) \otimes \theta (x)$ for any $x \in F$. For example, the Magnus expansion $\theta_{\exp} $ defined by $\theta_{\exp}(x_i)=1+ \sum_{j=1}^{\infty} X_i^j /j! $ is group-like. Furthermore, we define a set
$$\Theta_n^{\rm grp} :=  \{ \mathrm{Group} \textrm{-}\mathrm{like \ Magnus \ expansions \ }  \theta: F \ra (\hat{T}/\hat{T}_k)^{\times} \mathrm{\ over \ } A \mid A \in \GL_{n}(\K) \}. $$
For $ \theta \in \Theta_n^{\rm grp} $ and $ U \in \Aut_{\rm Hopf} ( \hat{T} /\hat{T}_k) $, we define $U \cdot \theta$ by $U \circ \theta $. From the definitions, $U \cdot \theta $ lies in $\Theta_n^{\rm grp}$. This correspondence defines an action of $\Aut_{\rm Hopf} ( \hat{T} /\hat{T}_k) $ on $ \Theta_n^{\rm grp}$. Similarly to \cite[Theorem 1.3 (2)]{Ka1}, we can prove the following:
\begin{prop}\label{prop7}
The action of $\Aut_{\rm Hopf} ( \hat{T}/\hat{T}_k ) $ on $ \Theta_n^{\rm grp}$ is free and transitive.
\end{prop}
\noindent
The proof is almost the same as the original one, so we will omit the proof here. As a result, we have a bijection between $\Aut_{\rm Hopf} ( \hat{T} /\hat{T}_k) $ and $ \Theta_n^{\rm grp} $. According to Lemma \ref{yyyy}, we can describe the group structure on $ \Theta_n^{\rm grp}$ via the bijectivity, where the identity $\in \Theta_n^{\rm grp} $ is the above $\theta_{\exp} $.

Next, let us briefly mention symplectic expansions, which were first introduced in \cite{Ma}. Let $n$ be $2g$, and $\zeta \in F$ be $[x_1,x_2] \cdots [x_{2g-1},x_{2g}]$, and $\omega \in \hat{T}/\hat{T}_k$ be $\sum_{i=1}^g X_{2i-1} X_{2i} - X_{2i} X_{2i -1}$. {\it A symplectic expansion} is a group-like expansion $\theta: F \ra (\hat{T}/\hat{T}_k)^{\times}$ satisfying $\theta (\zeta)=\exp(-\omega)= \sum_{m=0}^{\infty} (-\omega)^m/m! $. Let $ \Theta_n^{\rm sym} $ be a subset composed of all the symplectic expansion, and $\Aut_{\rm \omega} ( \hat{T} /\hat{T}_k) $ be the subgroup of $\Aut_{\rm Hopf} ( \hat{T} /\hat{T}_k) $ of automorphisms preserving $\omega$, which is isomorphic to the subgroup $ \Aut_{\omega} ( \mathcal{L} /\mathcal{L} _k)$ in \eqref{seq37}. As shown in \cite{Ma}, $ \Theta_n^{\rm sym}$ is not empty. As an extension of \cite[Proposition 2.8.1]{KK}, if $k = \infty$, we can show that the restricted action of $\Aut_{\rm \omega} ( \hat{T} /\hat{T}_k) $ on $ \Theta_n^{\rm sym}$ is also free and transitive.

Furthermore, let us review the total Johnson map from \cite[Section 2.5]{KK}. Let $\Ker(\epsilon) \subset \K[F]$ be the augmentation ideal, and $\widehat{\K}[F]$ the completion algebra, i.e., $\widehat{\K}[F] := \lim_{\infty \leftarrow j} \K[F]/(\Ker(\epsilon) )^j$. As in \cite[Theorem 1.3]{Ka1}, any Magnus expansion $\theta$ gives rise to a filtered algebra isomorphism $\theta: \widehat{\K}[F] \cong \widehat{T}$. Any mapping class $\varphi \in \mathcal{M}_{g,1}$ can be regarded as being in $\mathrm{Aut}(F)$ via the natural action $\mathcal{M}_{g,1} \curvearrowright \pi_1(\Sigma_{g,1})= F $; thus, we uniquely have $T^{\theta} (\varphi)\in \Aut_{\rm alg}(\widehat{T} )$ such that $T^{\theta} (\varphi) \circ \theta = \theta \circ \varphi $. The map $T^{\theta}: \mathcal{M}_{g,1} \ra \Aut_{\rm alg}(\widehat{T} )$ is also called {\it the total Johnson map}. It is known (see, e.g., \cite{KK, KK2}) that this map is injective, and if $\theta$ is group-like (resp. symplectic), then the image of $T^{\theta} $ is contained in $ \Aut_{\rm Hopf} ( \hat{T} ) $ (resp. $\Aut_{\rm \omega} ( \hat{T} ) $); the composite of the restriction of $T^{\theta} $ on $\mathcal{I}_{g,1} \cup \{\textrm{Dehn twists}\}$ and the logarithm from $\mathrm{IA}_{\rm Hopf} ( \hat{T} ) $ are studied in \cite{KK, KK2,Ma, MT} for the case where $\K=\Q$ and $\theta$ is over $A=I_n$.

In contrast, thanks to Theorem \ref{thm1}, when $\K=\C$, we defined the logarithms from the set
$$\mathrm{SE}_k:= \{ \phi \in \mathrm{Aut}_{\rm Hopf} ( \hat{T} /\hat{T}_k) \mid \mathrm{Eig}(p_k (\phi)) \cap \{ z \in \C \mid 1= |z|\}=1 \} .$$
To conclude, we have succeeded in giving an extension of the composite $\mathrm{Log} \circ T^{\theta}$ from the preimage $(T^{\theta})^{-1}(\mathrm{SE}_k) \subset \mathcal{M}_{g,1}$. It might be an interesting exercise to determine the union $\cup_{\theta \in \Theta_{n}^{\rm grp}} \cup_{k \geq 1}^{\infty} (T^{\theta})^{-1}(\mathrm{SE}_k) \subset \mathcal{M}_{g,1} $. For example, if this inclusion $\subset$ is surjective, we can define a logarithm for every element of $\mathcal{M}_{g,1} $.

Now let us discuss some small applications to topological invariants. The paper \cite{GL} defines a class of homology cylinders over $\Sigma_{g,1}$, for which we will not give a definition here, and a certain monoid homomorphism,
$$\sigma_k: \{ M : \textrm{A homology cobordism over }\Sigma_{g,1} \} \lra \Aut_{0}( \mathcal{L} /\mathcal{L}_k)\cong \Aut_{\omega}(\hat{T} /\hat{T}_k ), $$
where $\mathrm{rk} \mathcal{L}=2g$. As an analogy, the author \cite{Nos} defines a knot invariant as the map
$$\sigma_k' :\{ \textrm{a knot } K\textrm{ in }S^3 \textrm{ with } \mathrm{deg} \Delta_K=2g\} \lra \mathrm{Out}_{0}( \mathcal{L} /\mathcal{L}_k) /{\rm conj}, $$
where $\mathrm{deg} \Delta_K$ means the degree of the Alexander polynomial of $K$ and the symbol $``{\rm /conj}" $ means the set of the conjugacy classes. The groups $ \Aut_{\omega}(\hat{T} /\hat{T}_k )$ and $ \mathrm{Out}_{0}( \mathcal{L} /\mathcal{L}_k)$ seem far from being computable; however, if $\sigma_k(M) $ lies in $\mathrm{SE}_k$, the composite $\ln \circ \sigma_k(M)$ is valued in the vector space $\mathrm{Der}_{\omega} ( \mathcal{L}/ \mathcal{L}_k \otimes \C) $. On the other hand, concerning another ${\sigma}_k'$, notice the equality,
\begin{equation}\label{seq379}
X- e^{-Y}X e^{Y} = X-\sum_{j:j \geq 1} [[ \cdots [ X, \underbrace{Y] \cdots ],Y}_{j \textrm{-times}}]/j! \in \mathrm{Der} ( \mathcal{L}/ \mathcal{L}_k \otimes \C)
\end{equation}
for any $X,Y \in \mathrm{Der} ( \mathcal{L}/ \mathcal{L}_k \otimes \C)$; see, e.g., \cite[Appendix (A.0.2)]{MT}. In conclusion, if we can define $\ln \circ \sigma_k'(K)$, the invariant is valued in the quotient linear space $\mathrm{Der}_{\omega} ( \mathcal{L}/ \mathcal{L}_k \otimes \C) /\mathrm{Inn}_{\omega} ( \mathcal{L}/ \mathcal{L}_k \otimes \C) $ modulo the formula \eqref{seq379} for any $X,Y \in \mathrm{Der}_{\omega} ( \mathcal{L}/ \mathcal{L}_k \otimes \C)$. Here, $\mathrm{Inn}_{\omega} ( \mathcal{L}/ \mathcal{L}_k \otimes \C) $ is the submodule consisting of inner derivations killing $ \omega$. Additionally, we remark that the quotient space modulo \eqref{seq379} is non-trivial. In fact, the space surjects on the abelianization of $\mathrm{Der}_{\omega} ( \mathcal{L}/ \mathcal{L}_k \otimes \C) $ which is related to the rational cohomology of $\mathrm{Out}(F)$; see the main theorem of \cite{Kon} and \cite[Section 8]{CKV}. For example, if $g \geq 6$ or $g=1$, the paper \cite{MSS} explicitly gives 1-cocycles of $\mathrm{Der}_{\omega} ( \lim_{ \infty \leftarrow k} \mathcal{L}/ \mathcal{L}_k \otimes \C) $ in terms of the Enomoto-Satoh trace. 

As in the Kontsevich knot invariant and the LMO invariant of homology 3-spheres, if we can define the logarithms of some invariants, we should be able to discover interesting properties and relations between invariants from the viewpoint of the logarithms. It seems reasonable to hope that the above composites $\ln \circ \sigma_k$ and $ \ln \circ \sigma_k'$ have similar properties.

\section{Two lemmas }
\label{333}
Here, we turn our attention to the proofs of the main theorems. As a preliminary, let us prepare two lemmas. Proposition \ref{prop4} below is inspired by \cite[\S 2]{Mo} and \cite{Mo2}. We will use the terminology in Section \ref{99414}. For $k,m \in \mathbb{N}$, notice from Lemma \ref{thm130} that the canonical projection $q_m: \mathrm{Aut}_{\bullet }( \hat{T}/\hat{T}_{m+1} ) \ra \mathrm{Aut}_{\bullet }(\hat{T}/\hat{T}_m ) $ is surjective, where $\bullet$ means either alg or Hopf. Recall from \eqref{7577} that any $\K$-homomorphism $\alpha: H_{\K} \ra \hat{T}_{m}/\hat{T}_{m+1}$ such that $p_k \circ \alpha: H_{\K} \ra H_{\K} $ is an isomorphism can be uniquely extended to an algebra isomorphism $ \tilde{\alpha}: \hat{T}/\hat{T}_{m+1} \ra \hat{T}/\hat{T}_{m+1} $; thus, the kernel of $q_m$ is $ \Hom_{\K} ( H_{\K}, \hat{T}_{m}/\hat{T}_{m+1} )$. Notice that $\mathrm{Aut}_{\bullet }( \hat{T}/\hat{T}_{m} ) $ is a closed subgroup of the matrix Lie group $ \mathrm{GL}( \hat{T}/\hat{T}_{m+1} ) $ and $q_m$ is continuous; hence, if $\K=\R $ or $\K =\C$, then $\mathrm{Aut}_{\bullet }( \hat{T}/\hat{T}_{m} ) $ is a Lie group and $q_m$ is smooth. In summary, we have the Lie group extensions,
\begin{equation}\label{seq388}
{\normalsize \xymatrix{
0\ar[r] & \Hom_{\K} ( H_{\K} , \hat{T}_{m}/\hat{T}_{m+1} ) \ar[r]^{E} & \Aut_{\rm alg}(\hat{T}/\hat{T}_{m+1}) \ar[r]^{q_m}& \Aut_{\rm alg}(\hat{T}/\hat{T}_{m}) \ar[r] & 1\\
0\ar[r] & \Hom_{\K} ( H_{\K} , \mathcal{P}(\hat{T}_{m}/\hat{T}_{m+1} )) \ar[r]^{E} \ar@{^{(}-_{>}}[u] & \Aut_{\rm Hopf}(\hat{T}/\hat{T}_{m+1}) \ar[r]^{q_m} \ar@{^{(}-_{>}}[u]& \Aut_{\rm Hopf}(\hat{T}/\hat{T}_{m}) \ar[r] \ar@{^{(}-_{>}}[u]& 1.}
}\end{equation}
Here, $E$ is the restriction of the map defined in \eqref{7577}.
The sequence \eqref{seq388} over $\Q$ follows from Proposition 2.3 in \cite{Mo} and is not central; however, to see a centrality, we only need to recall the subgroups in \eqref{777}; we have
% $I( \mathfrak{n}_k^{\K} )$ of $\Aut( \mathfrak{n}_k^{\K} )$ to be the kernel of the projection \begin{equation}\label{seq44}p_m : \Aut_{\rm Lie}( \mathfrak{n}_k^{\K} ) \lra \Aut_{\rm Lie}( \mathfrak{n}_1^{\K} ) =\GL(H_\K).\end{equation}
\begin{lem}\label{prop4}
The horizontal sequences of the restriction of the diagram \eqref{seq388},
\begin{equation}\label{seq4}
{\normalsize \xymatrix{
0\ar[r] & \Hom_{\K} ( H_{\K} , \hat{T}_{m}/\hat{T}_{m+1} ) \ar[r]^{E} & \mathrm{IA}_{\rm alg}(\hat{T}/\hat{T}_{m+1}) \ar[r]^{q_m}& \mathrm{IA}_{\rm alg}(\hat{T}/\hat{T}_{m}) \ar[r] & 1\\
0\ar[r] & \Hom_{\K} ( H_{\K} , \mathcal{P}(\hat{T}_{m}/\hat{T}_{m+1} )) \ar[r]^{E} \ar@{^{(}-_{>}}[u] & \mathrm{IA}_{\rm Hopf}(\hat{T}/\hat{T}_{m+1}) \ar[r]^{q_m} \ar@{^{(}-_{>}}[u]& \mathrm{IA}_{\rm Hopf}(\hat{T}/\hat{T}_{m}) \ar[r] \ar@{^{(}-_{>}}[u]& 1,}
}\end{equation}
are central extensions. In particular, $\mathrm{IA}_{\rm \bullet}(\hat{T}/\hat{T}_{m})$ is a nilpotent Lie group and is contractible.
\end{lem}
\begin{proof} %[Proof of Proposition \ref{prop4}]
We will show the centrality of only the upper sequence. For $f \in \Hom ( H_{\K} , \hat{T}_{m}/\hat{T}_{m+1} )$ and $\alpha \in \mathrm{IA}_{\rm alg}(\hat{T}/\hat{T}_{m+1}) $, by \eqref{777}, it suffices to show the identity $ ( \alpha \circ E(f))^{\sim }= (E(f) \circ \alpha )^{\sim }$. Let $r_k: \hat{T}/\hat{T}_{m+1}\ra \hat{T}_{m}/\hat{T}_{m+1} $ be the projection. Then, $ r_m(( \alpha \circ E(f))^{\sim })=r_m( \alpha )^{\sim } + r_m( E(f))^{\sim } =r_m( (E(f) \circ \alpha )^{\sim }) $ by the definition of $\mathrm{IA}_{\rm alg}(\hat{T}/\hat{T}_{m+1}) $. Since $q_m( \alpha \circ E(f)) = q_m( \alpha)=q_m(E(f) \circ \alpha)$ by the definition of $f$, we have the required identity.
\end{proof}

Next, let us change our focus to exponential solvable Lie groups. Let $G$ be a real solvable Lie group with Lie algebra $\mathfrak{g}$. If the exponential map $ \mathfrak{g} \ra G$ is diffeomorphic, $G$ is said to be {\it exponential solvable}. For example, every simply-connected nilpotent Lie group is exponentially solvable. As an application, let us consider the situation that the additive Lie group $\R$ smoothly acts on a Lie group $N$ with Lie algebras $ \mathfrak{n} $. Then, the action defines the semi-direct product $ N\rtimes \R$ and induces a representation $ \tau :\R \ra \mathrm{GL}(\mathfrak{n})$. For $x \in \mathfrak{n} $, let $T(x) \subset \R$ be the subgroup of $\R$ fixing $x$. We shall cite the following: 
\begin{lem}
[{A special case of \cite[Theorem 5]{MM}}]\label{thm10} Let $N$ be a simply-connected nilpotent Lie group and $G = N \rtimes \R$ be the semidirect product associated with $\tau$. Then, $G$ is exponential solvable if and only if $T(x) =\{0\}$ or $T(x)=\R$ for each $x \in \mathfrak{n} $.
\end{lem} 
\begin{rem}\label{rem10}
The Jordan decomposition theorem implies that the condition concerning $T(x)$ in Lemma \ref{thm10} holds if no eigenvalue of $ \tau (1) \in \mathrm{GL}(\mathfrak{n})$ lies in $\{ b \in \C^{\times} \mid b\neq 1,|b|=1 \}$. This is equivalent to saying that no eigenvalue of the differential $d \tau (\bullet) \in \mathrm{End}(\mathfrak{n} \otimes {\C})$ lies in $\{ \sqrt{-1} a \mid a \in \R^{\times}\}$ (cf. the concept of weights on exponential solvable Lie groups; see, e.g., \cite{MM}). 
\end{rem}

\section{Proofs of Theorems \ref{thm1}, \ref{thm2}, \ref{thm7}, and \ref{thm51}}
\label{99se33}
Before giving the proofs, we will briefly examine the decomposition \eqref{pp9} below. For $\lambda \in \C \setminus \{0\}$ and $ \ell \in \N$, let $J_{\lambda}(\ell) \in \mathrm{Mat}( \ell \times \ell; \C)$ be the Jordan block of size $\ell$ and eigenvalue $\lambda $. As is known (see, e.g., \cite[Theorem 2]{MV}), the tensor product of $ J_{\lambda}(\ell)$ and $J_{\mu}(n) $ in $\mathrm{Mat}( n\ell \times n\ell; \C) $ has the following decomposition:
\begin{equation}\label{pp9} J_{\lambda}(\ell) \otimes_{\C} J_{\mu}(n) \cong \bigoplus_{w: 1 \leq w \leq \mathrm{min}(\ell,n)} J_{ \lambda \mu}(\ell + n- 2w-1) . \end{equation}

In addition, let us observe the action of $ \GL(H_{\K}) $ on the kernel $ \mathrm{IA}_{\bullet}(\hat{T}/\hat{T}_{k})$ in \eqref{777}. By Lemma \ref{thm130}, the Lie algebra of $\mathrm{IA}_{\rm alg} (\hat{T}/\hat{T}_{k}) $ is linearly isomorphic to $\oplus_{j=2}^{k-1}\Hom_{\K}( H_{\K},\hat{T}_{j}/\hat{T}_{j+1} )$. Here, we consider the action of $ \GL(H_{\K} )$ on $\Hom( H_{\K}, \hat{T}_{j}/\hat{T}_{j+1} )$ defined by 
\begin{equation}\label{seq8}
(Af)(u) := A^{\otimes j }f(A^{-1} u) \ \ \ \ \ \ \ (A \in \GL( H_{\K}), \ u \in H_{\K}, \ f \in \Hom( H_{\K}, \hat{T}_{j}/\hat{T}_{j+1} )).
\end{equation}
From Lemma \ref{yyyy} and the semi-direct products \eqref{777}, the sum of this action \eqref{seq8} coincides with the action of $ \GL(H_{\K}) $ on $\mathrm{IA}_{\rm alg}(\hat{T}/\hat{T}_{k})$, although this coincidence was shown in the case $\K=\Q$ (see, e.g., \cite{Mo,Mo2}). Hence, the restricted action of $ \GL(H_{\K} )$ on $ \mathrm{IA}_{\rm Hopf }(\hat{T}/\hat{T}_{k}) = \oplus_{j=1}^{k-1} \Hom( H_{\K}, \mathcal{P}(\hat{T}_{j}/\hat{T}_{j+1} ) )$ can be regarded as a subrepresentation of $\oplus_{j=2}^{k-1}H_{\K}^* \otimes H^{\otimes j}_{\K} $.
\begin{proof}
[Proof of Theorem \ref{thm1}] Let $\Phi \in \Aut_{\rm Lie}( \mathcal{L}/\mathcal{L}_{k}) \cong \Aut_{\rm Hopf}( \hat{T}/\hat{T}_{k})$, and let $\K$ be $\C$. If $p_k(\Phi)$ is the identity matrix, then $\phi$ is the zero matrix and $ \Phi $ lies in $ \mathrm{IA}_{\rm Hopf} (\hat{T}/\hat{T}_{m}) $. Since any contractible nilpotent Lie group is exponential solvable as mentioned in Section \ref{333}, so is $\mathrm{IA}_{\rm Hopf} (\hat{T}/\hat{T}_{m}) $, by Lemma \ref{prop4}. Hence, it is enough to define $D_{\Phi}$ to be $\exp^{-1} (\Phi)$. 

Next, we suppose that $p_k(\Phi) $ is not the identity. Let $ K$ be the 1-dimensional Lie subgroup of $ \mathrm{GL}( H_{\C})$ generated by $p_k(\Phi)$, that is, $K= \{ \exp(t \phi) \}_{t \in \R}$. From the definition of the exponential solvability of $p_k(\Phi)$, $K$ is diffeomorphic to $\R$ in $\mathrm{GL}( H_{\C}). $ The restricted action from the semi-direct product \eqref{777} is regarded as a representation $\mathfrak{s}: K \ra \mathrm{IA}_{\rm Hopf }(\hat{T}/\hat{T}_{k}) \subset \GL(\hat{T}/\hat{T}_{k}) $, which the above discussion shows is a subrepresentation of $\R$ on $\oplus_{j=2}^{k-1} H_{\C}^* \otimes H^{\otimes j}_{\C} $. Notice from the decomposition \eqref{pp9} that any eigenvalue of $\mathfrak{s}(\exp(\phi) ) $ lies in $ \mathrm{Eig}( p_k(\Phi) ) $, and, therefore, does not intersect the set $\{ b \in \C^{\times} \mid b\neq 1,|b|=1 \}$ under the assumption of $ p_k(\Phi)$. From Remark \ref{rem10}, the action of $K $ on the Lie algebra of $\mathrm{IA}_{\rm Hopf }(\hat{T}/\hat{T}_{k}) $ satisfies the condition in Lemma \ref{thm10}; thus, $\mathrm{IA}_{\rm Hopf }(\hat{T}/\hat{T}_{k})\rtimes K $ is an exponential solvable Lie group and contains $\Phi$. Hence, if we define $D_{\Phi}$ to be $\exp^{-1} (\Phi)$, then $D_{\Phi}$ satisfies the required conditions $\exp(D_\Phi )=\Phi $ and $ p_k(D_\Phi )= \phi$ by construction. Finally, uniqueness follows from the bijectivity of $\exp$.
\end{proof}

\begin{proof}[Proof of Theorem \ref{thm51}] If we replace $\Aut_{\rm Hopf} $ by $\Aut_{\rm alg}$ in the above proof,
the same discussion runs. Thus, we omit the details. 
\end{proof} 

Next, to prove Theorem \ref{thm2}, we review some of the results from \cite{Mo2}. Let $\mathrm{rk}(\mathcal{L})=2g$, and $ \mathrm{IA}(\mathcal{L}/ \mathcal{L}_{k} \otimes \K)$ be $\Ker(p_k) \cap \mathrm{Aut}_{\rm Lie}(\mathcal{L}/ \mathcal{L}_{k} \otimes \K) $. Consider the kernel of the bracketing:
\begin{equation}\label{seq63}
\mathfrak{h}_{g,1}^{\K}(k):=\Qer([,]: H_{\K} \otimes_{\Q} (\mathcal{L}_{k}/ \mathcal{L}_{k+1}) \lra (\mathcal{L}_{k+1}/ \mathcal{L}_{k+2})\otimes_{\Q} \K) ,
\end{equation}
Under the identification $H_{\K} \cong H_{\K }^* $ as an Sp-module, 
let us identify $\Hom (H_{\K} , \mathcal{L}_{k}/ \mathcal{L}_{k+1}\otimes \K)$ with $H_{\K} \otimes_{\Q} \mathcal{L}_{k}/ \mathcal{L}_{k+1} $;
%Under the identification of $\Hom (H_{\K} , \mathcal{L}_{k}/ \mathcal{L}_{k+1}\otimes \K)$ and $H_{\K} \otimes_{\Q} \mathcal{L}_{k}/ \mathcal{L}_{k+1} $ by the Poincar\'{e} duality $H_{\K} \cong H_{\K }^* $, 
as shown in \cite[(3)]{Mo2}, if $\K=\Q$, then the intersection $\mathrm{Im}(E) \cap \Aut_0( \mathcal{L}/ \mathcal{L}_{k+1} \otimes \K) $ is equal to the $\mathfrak{h}_{g,1}^{\K}(k)$. Further, the equality holds over $\K=\C$, since $\mathrm{IA}(\mathcal{L}/ \mathcal{L}_{k+1} \otimes \Q) \subset \mathrm{IA}(\mathcal{L}/ \mathcal{L}_{k+1} \otimes \R) $ is dense, and $ \mathrm{IA}(\mathcal{L}/ \mathcal{L}_{k+1} \otimes \R) \subset \mathrm{IA}(\mathcal{L}/ \mathcal{L}_{k+1} \otimes \C) $ is a canonical complexification. Meanwhile, any $f \in \Hom (H_{\K} , \mathcal{L}_{k}\otimes \K) \cong H_{\K} \otimes (\mathcal{L}_{k}\otimes \K) $ can be uniquely extended to a derivation $D_f$ in $\mathrm{Der}(\mathcal{L}/ \mathcal{L}_{k} \otimes \K)$; we have an injection $ \oplus_{j=0}^{k-1} \mathfrak{h}_{g,1}^{\K}(j) \hookrightarrow \mathrm{Der}(\mathcal{L}/ \mathcal{L}_{k} \otimes \C)$ that sends $f$ to $D_f$, whose image is a Lie subalgebra. According to \cite[Theorem 3.3]{Mo2}, if $\K=\Q$, the Lie algebras of $ \Aut_0(\mathcal{L}/ \mathcal{L}_{k} \otimes \K) $ and $ \Aut_0(\mathcal{L}/ \mathcal{L}_{k} \otimes \K)\cap \mathrm{IA}(\mathcal{L}/ \mathcal{L}_{k} \otimes \K) $ are isomorphic to the subalgebras $ \oplus_{j=0}^{k-1} \mathfrak{h}_{g,1}^{\K}(j) $ and $\oplus_{j=1}^{k-1} \mathfrak{h}_{g,1}^{\K}(j) $, respectively. The original statement is given only over $\Q$; however, we can similarly verify that the isomorphisms hold over $\R$ and $\C$.
\begin{proof}
[Proof of Theorem \ref{thm2}] Since $ \mathfrak{h}_{g,1}^{\C}(k) = \mathrm{Im}(E) \cap \Aut_0( \mathcal{L}/ \mathcal{L}_{k} \otimes \C ) $ as above, the extension \eqref{seq4} means that $ \mathrm{IA}(\mathcal{L}/ \mathcal{L}_{k} \otimes \C)\cap \Aut_0( \mathcal{L}/ \mathcal{L}_{k} \otimes \C) $ is contractible. Since it is nilpotent, it is an exponential solvable Lie group. Hence, the same discussion as in the proof of Theorem \ref{thm1} applies if we replace $\mathrm{IA}_{\rm Hopf }(\hat{T}/\hat{T}_{k}) $ by $ \mathrm{IA}(\mathcal{L}/ \mathcal{L}_{k} \otimes \C)\cap \Aut_0( \mathcal{L}/ \mathcal{L}_{k} \otimes \C) $. 
\end{proof}

It remains to give the proofs of Lemma \ref{yyy} and Theorem \ref{thm7}:
\begin{proof}
[Proof of Lemma \ref{yyy}] Let $k \in \mathbb{N}$ be arbitrary. We may suppose that $p_k(\Phi) $ is not the identity. As in the above proof, consider the exponential solvable Lie group $\mathrm{IA}_{\rm Hopf }(\hat{T}/\hat{T}_{k}) \rtimes K $. In general, it is known (see, e.g., \cite[Lemma 1.1 or Sections 1.3 and 1.9]{Ro}) that the BCH formula in any exponential solvable Lie group $G$ converges for any element of the Lie algebra $\mathfrak{g}$. Hence, the BCH formula \eqref{xyx} is in the Lie algebra of $\mathrm{IA}_{\rm Hopf }(\hat{T}/\hat{T}_{k}) \rtimes K$ converges for any $ X,Y$.
\end{proof}
\begin{proof}
[Proof of Theorem \ref{thm7}] Let $\Phi \in \Aut_{\rm Lie}( \mathcal{L}/\mathcal{L}_{k}) \cong \Aut_{\rm Hopf}( \hat{T}/\hat{T}_{k})$ such that every eigenvalue of $p_k(\Phi)$ is one. If we can show the existence of $N_k \in \N$ satisfying $ ( \Phi - \mathrm{id})^{N_k} =0 $ with $\K=\C$, then the logarithm $\mathrm{Log}$ converges, while $\ln(\Phi)= \mathrm{Log} (\Phi) $ immediately follows from the uniqueness in Theorem \ref{thm1}. Here, the statement in the case $k= \infty$ is due to the inverse limit of $\ln(\Phi)= \mathrm{Log} (\Phi) $ according to $k \ra \infty.$

As for the proof of the existence of $N_k$, let $\phi \in \mathfrak{gl}( \hat{T}/\hat{T}_{2} )$ be $ \mathrm{Log}(p_k(\Phi))= -\sum_{i=1}^{n-1}( p_k(\Phi) - I_n)^i/i$. Let $ \mathfrak{IA} \subset \mathfrak{gl}( \hat{T}/\hat{T}_{k} )$ and $\mathfrak{k} = \langle \phi \rangle $ be the Lie algebras of the nilpotent Lie group $ \mathrm{IA}_{\rm alg }(\hat{T}/\hat{T}_{k}) $ and of $ K= \langle p_k(\Phi)\rangle $, respectively. We claim that any $(Y,\phi) \in \mathfrak{IA} \rtimes \mathfrak{k}$ ensures $n_Y \in \mathbb{N}$ such that $(Y,\phi )^{n_Y} v=0$ for any $v \in \hat{T}/\hat{T}_{k} $. By the decomposition \eqref{pp9}, if $ d \geq \dim \hat{T}/\hat{T}_{k}$, then $ (0,\phi)^d (\hat{T}/\hat{T}_{k})=0$. Thus, by the Poincar\'{e}-Birkhoff-Witt theorem, we have that $ X \in \mathfrak{IA} $ satisfies $(Y,\phi )^d v= X v$ for any $v \in \hat{T}/\hat{T}_{k}. $ Notice from Engel's theorem on the nilpotent $\mathfrak{IA} $ (or Lemma \ref{ppo2}) that any $Z \in \mathfrak{IA} $ admits $n_Z \in \mathbb{N}$ such that $Z^{n_Z} (\hat{T}/\hat{T}_{k}) =0$; hence, to complete our proof of the claim, we may define $n_Y$ to be $n_Z d$. Then, from this claim and Lemma \ref{ppo2} below, we can choose an appropriate basis of $ \hat{T}/\hat{T}_{k} $ such that the semi-direct product $\mathrm{IA}_{\rm alg }(\hat{T}/\hat{T}_{k}) \rtimes K$ is a Lie subgroup of $U( \hat{T}/\hat{T}_{k})$, where $U( \hat{T}/\hat{T}_{k})$ is the subgroup of $ \GL( \hat{T}/\hat{T}_{k} )$ consisting of upper right triangular matrices with only eigenvalue one. Since $\Phi \in \mathrm{IA}_{\rm alg }(\hat{T}/\hat{T}_{k}) \rtimes K \subset U( \hat{T}/\hat{T}_{k})$, we have $(\Phi - \mathrm{id})^{\dim \hat{T}/\hat{T}_{k}}(\hat{T}/\hat{T}_{k}) =0$, as required.
\end{proof}
\begin{lem}
[{A corollary of Engel's theorem. See, e.g., \cite[Corollary 11.11]{SW}}]\label{ppo2}Let $V$ be a real vector space with $ 1 \leq \dim_{\R } V < \infty$, and $L$ be a Lie subalgebra of $\mathfrak{gl}( V) $. If any $X \in L$ admits $n_X \in \mathbb{N}$ satisfying $ X^{n_X} V=0$, then we can choose a basis $\{ v_i\}_{ i=1}^{\dim V}$ of $V$ such that $L$ is contained in $\mathfrak{n}( V)$. Here, $\mathfrak{n}( V)$ is the Lie subalgebra $\{ (n_{ij})_{ 1 \leq i,j \leq \dim V }\in \mathfrak{gl}( V)= \mathrm{Mat}(\dim V \times \dim V ;\R ) \mid n_{ij} =0 \ \ (i \geq j)\}$.
\end{lem}
%\begin{proof}This is an immediate result of Engel's theorem (or the proof); see, e.g., \cite[Corollary 11.11]{SW}.%, we have a non-zero $v_1 \in V$ 
% \end{proof}

\subsection*{Acknowledgments}
The author sincerely expresses his gratitude to Nariya Kawazumi, Yusuke Kuno, and Masatoshi Sato for giving him valuable comments. 
He gratefully acknowledges many helpful suggestions of the referee.
%reatly indebted to the referee for carefully reading this paper and pointing out 
The work was partially supported by JSPS KAKENHI, Grant Number 00646903.

\vskip 1pc

\normalsize

\noindent
Department of Mathematics, Tokyo Institute of Technology
2-12-1
Ookayama, Meguro-ku Tokyo 152-8551 Japan

\end{document}